\documentclass[12pt]{article}
\usepackage{amsmath, amsthm}
\usepackage[top=4cm,left=3cm,right=3cm,bottom=4cm,headheight=13.6pt]{geometry}\usepackage{amsfonts}
\usepackage{amssymb}
\usepackage{fancyhdr}
\usepackage{enumerate}
\usepackage{amsmath,amsthm}
\usepackage[dvips]{graphicx}
\usepackage{color}
\usepackage{eepic}
\usepackage{epic}
\usepackage[dvips]{rotating}
\usepackage{sectsty}
\usepackage{wrapfig}

\theoremstyle{plain}
\newtheorem{theorem}{Theorem}[section]
\newtheorem{lemma}[theorem]{Lemma}

\theoremstyle{definition}

\theoremstyle{remark}

 \numberwithin{equation}{section} 

\begin{document}
\title{The Blow-up  Rate Estimates for  a System of Heat Equations  with Nonlinear Boundary Conditions}
\author{Maan A. Rasheed and Miroslav Chlebik}

\maketitle

\abstract 

This paper deals with  the  blow-up properties of positive solutions to a system of two heat equations  $u_t= \Delta u,~v_t=\Delta v$ in $B_R  \times (0,T)$ with Neumann boundary conditions 
$\frac{\partial u}{\partial \eta} =e^{v^p},~ \frac{\partial v}{\partial \eta}=e^{u^q}$ on $\partial B_R  \times (0,T),$ where $p,q>1,$ $B_R$ is a ball in $R^n,$ $\eta$ is the outward normal. The upper  bounds of blow-up rate estimates were obtained. It is also proved that the blow-up occurs only on the boundary.

\section{\bf Introduction}

In this paper, we consider the system of  two heat equations with coupled nonlinear Neumann boundary conditions, namely
\begin{equation}\label{c1}\left. \begin{array}{lll}
u_t= \Delta u,& v_t=\Delta v,& (x,t)\in B_R  \times (0,T),\\
\frac{\partial u}{\partial \eta} =e^{v^p}, & \frac{\partial v}{\partial \eta}=e^{u^q},& (x,t)\in \partial B_R  \times (0,T),\\
u(x,0)=u_0(x),&  v(x,0)=v_0(x),&  x \in {B}_R,
\end{array} \right\} \end{equation} 
   where  $p,q>1,$ $B_R$ is a ball in $R^n,$ $\eta$ is the outward normal, $u_0,v_0$ are smooth, radially symmetric, nonzero, nonnegative functions satisfy the condition
   \begin{equation}\label{Mo}\Delta u_0,\Delta u_0\ge 0,\quad u_{0r}(|x|),v_{0r}(|x|)\ge 0,\quad  x  \in \overline{B}_R. \end{equation}

The problem of system of two heat equations with nonlinear Neumann boundary conditions defined in a ball, 

\begin{equation}\label{3}\left. \begin{array}{lll}
u_t= \Delta u,& v_t=\Delta v,& (x,t)\in B_R  \times (0,T),\\
\frac{\partial u}{\partial \eta} =f(v), & \frac{\partial v}{\partial \eta}=g(u),& (x,t)\in \partial B_R  \times (0,T),\\
u(x,0)=u_0(x),&  v(x,0)=v_0(x),&  x \in {B}_R,
\end{array} \right\} \end{equation} 

was introduced in \cite{19,7,20,18}, for instance, in \cite{19} it was studied the blow-up solutions to the system (\ref{3}), where
\begin{equation}\label{c11}
 f(v) =v^p, \quad g(u) =u^q, \quad p,q>1.
  \end{equation}
It  was proved that for any nonzero, nonnegative initial data $(u_0,v_0),$ the finite time blow-up can only occur on the boundary, moreover, it was shown in \cite{20} that, the blow-up rate estimates take the following form
$$c\le \max_{x\in\overline{\Omega}}u(x,t)(T-t)^{\frac{p+1}{2(pq-1)} }\le C,\quad t\in(0,T),$$
$$c\le\max_{x\in\overline{\Omega}}v(x,t)(T-t)^{\frac{q+1}{2(pq-1)} }\le C,\quad t\in(0,T).$$

In \cite{7,18}, it was considered the solutions of the system (\ref{3}) with exponential Neumann boundary conditions model, namely
\begin{equation}\label{c12}
f(v) =e^{pv}, \quad  g(u) =e^{qu},\quad p,q>0. \end{equation}
It was proved that for any nonzero, nonnegative initial data, $(u_0,v_0),$ the solution blows up in finite time and the blow-up occurs only on the boundary, moreover, the blow-up rate estimates take the following forms
\begin{equation*}
C_1\le e^{qu(R,t)}(T-t)^{1/2} \le C_2, \quad C_3\le e^{pv(R,t)}(T-t)^{1/2} \le C_4.
\end{equation*}

In this paper, we prove that the upper blow-up rate estimates for problem (\ref{c1}) take the following form
\begin{align*}
\max_{ \overline{B}_R}u(x,t)\le \log C_1-\frac{\alpha}{2} \log(T-t),\quad 0<t<T, \\ \max_{ \overline{B}_R}v(x,t)\le \log C_2-\frac{\beta}{2} \log (T-t),\quad 0<t<T,  \end{align*}  where 
$\alpha=\frac{p+1}{pq-1},\beta=\frac{q+1}{pq-1}.$
 Moreover, the blow-up occurs only on the boundary.

\section{Preliminaries}
The local existence and uniqueness of classical solutions to problem (\ref{c1}) is well known by  \cite{47}. On the other hand, every nontrivial solution blows up simultaneously in finite time, and that due to the known blow-up results of problem (\ref{3}) with (\ref{c11}) and the comparison principle \cite{47}.  
  
 In the following lemma we study some properties of the classical solutions of problem (\ref{c1}). We denote for simplicity $u(r,t)=u(x,t).$
\begin{lemma}\label{c}
    Let $(u,v)$ be a classical unique solution of (\ref{c1}). Then
\begin{enumerate}[\rm(i)]
\item $u, v$ are  positive, radial. Moreover, $u_r ,v_r \ge 0$ in $[0,R]\times (0,T).$
 \item   
$u_t,v_t>0$ in $\overline B_R\times (0,T).$ 
\end{enumerate}
\end{lemma}

\section{Rate Estimates}
In order to study the upper blow-up rate estimates for problem (\ref{c1}), we need to recall some results from \cite{23,20}.
\begin{lemma}{\bf \cite{20}}\label{d}
Let  $A(t)$ and $B(t)$ be positive $C^1$ functions in $[0,T)$ and satisfy
$$A^{'} (t)\ge c \frac{B^p(t)}{\sqrt{T-t}},\quad B^{'} (t)\ge c \frac{A^q(t)}{\sqrt{T-t}} \quad \mbox{for} \quad t \in [0,T),$$
$$A(t) \longrightarrow + \infty \quad \mbox{or} \quad B(t) \longrightarrow + \infty \quad \mbox{as} \quad t  \longrightarrow T^{-},$$
where $p,q>0,c>0 $ and $pq>1.$ Then there exists $C>0$ such that
$$A(t) \le C(T-t)^{-\alpha /2}, \quad B(t) \le C(T-t)^{-\beta /2}, \quad t \in[0,T),$$
where $\alpha=\frac{p+1}{pq-1},\beta=\frac{q+1}{pq-1}.$
\end{lemma}
\begin{lemma}\label{pk}{\bf \cite{23}}
Let $x \in \overline{B}_R.$ If $0\le a <n-1.$ Then there exist $C>0$ such that $$\int_{S_R} \frac{ds_y}{|x-y|^a }\le C.$$
\end{lemma}
\begin{theorem}\label{Jump}{\bf (Jump relation, \cite{23})}
Let $\Gamma(x,t)$ be the fundamental solution of heat equation, namely
\begin{equation}\label{op}\Gamma(x,t)=\frac{1}{(4\pi t)^{(n/2)}}\exp[-\frac {|x|^2}{4t}] .\end{equation} Let $\varphi$ be a continuous function on $S_R \times[0,T].$ Then for any $x \in {B}_R,x^0 \in S_R,0<t_1<t_2\le T,$ for some $T>0,$ the function
$$U(x,t)=\int_{t_1}^{t_2} \int_{S_R} \Gamma (x-y,t-z)\varphi(y,z)ds_yd\tau  $$ satisfies the jump realtion
$$\frac{\partial}{\partial \eta}U(x,t) \rightarrow -\frac{1}{2}\varphi(x^0,t)+\frac{\partial}{\partial \eta}U(x^0,t), \quad as\quad x \rightarrow x^0.$$
\end{theorem}
\begin{theorem}\label{theorem d}
Let $(u,v)$ be a solution of (\ref{c1}), which blows up in finite time T. Then there exist positive constants $C_1,C_2$  such that
 \begin{align*}
\max_{ \overline{B}_R}u(x,t)\le \log C_1-\frac{\alpha}{2} \log(T-t),\quad 0< t<T, \\ \max_{ \overline{B}_R}v(x,t)\le \log C_2-\frac{\beta}{2} \log (T-t),\quad 0< t<T. \end{align*} 
\end{theorem}
\begin{proof}
We follow the idea of \cite{20}, define the functions $M$ and $M_b$ as follows
$$M(t)=\max_{\overline{B}_R}u(x,t), \quad \mbox{and}\quad M_b(t)=\max_{S_R}u(x,t).$$ Similarly, 
 $$N(t)=\max_{ \overline{B}_R}v(x,t), \quad \mbox{and}\quad N_b(t)=\max_{S_R}v(x,t).$$
 Depending on Lemma \ref{c}, both of $M,M_b$ are monotone increasing functions, and
  since $u$ is a solution of heat equation, it cannot attain interior maximum without being constant,
therefore, $$M(t)=M_b(t).\quad \mbox{Similarly}\quad N(t)=N_b(t).$$
Moreover, since $u,v$ blow up simultaneously, therefore, we have 
 \begin{equation}\label{Nobal} M(t)\longrightarrow +\infty, \quad N(t)\longrightarrow +\infty  \quad \mbox{as} \quad t\longrightarrow T^{-}.\end{equation}As in \cite{22,20}, for $0<z_1<t<T $ and $ x\in B_R,$ depending on the second Green's identity with assuming the Green function: $$G(x,y;z_1,t)=\Gamma(x-y,t-z_1),$$ where $\Gamma$ is defined in (\ref{op}), the integral equation to problem (\ref{c1}) with respect to $u,$ can be written as follows
\begin{eqnarray*}u(x,t)&=&\int_{B_R} \Gamma (x-y,t-z_1)u(y,z_1)dy+\int_{z_1}^t \int_{S_R}e^{v^p(y,\tau)}
\Gamma (x-y,t-\tau)ds_y d_\tau\\   &&-\int_{z_1}^t \int_{S_R}
u(y,\tau ) \frac{\partial \Gamma}{\partial \eta _y} (x-y,t-\tau )ds_y d \tau,\end{eqnarray*}
As in \cite{22}, letting $x\rightarrow S_R$ and using the jump relation (Theorem \ref{Jump}) for the third term on the right hand side of the last equation, it follows that
\begin{eqnarray*}\frac{1}{2}u(x,t)&=&\int_{B_R} \Gamma (x-y,t-{z_1})u(y,z_1)dy+\int_{z_1}^t \int_{S_R}e^{v^p(y,\tau)}
\Gamma (x-y,t-\tau)d{s_y}d_\tau\\  &&-\int_{z_1}^t \int_{S_R}
u(y,\tau ) \frac{\partial \Gamma}{\partial \eta _y} (x-y,t-\tau )ds_y d \tau,\end{eqnarray*} for $x\in S_R, 0<z_1<t<T.$

Depending on Lemma \ref{c} we notice that $u,v$  are positive and radial.Thus \begin{eqnarray*}&&\int_{B_R} \Gamma (x-y,t-z_1)u(y,z_1)dy >0,\\  &&\int_{z_1}^t \int_{S_R}e^{v^p(y,\tau)}
\Gamma (x-y,t-\tau)d{s_y}d_\tau=\int_{z_1}^t e^{v^p(R,\tau)}[\int_{S_R}\Gamma (x-y,t-\tau)ds_y]d\tau. \end{eqnarray*} This leads to
\begin{eqnarray*}\frac{1}{2}M(t) &\ge& \int_{z_1}^t e^{N^p(\tau)}[\int_{S_R}\Gamma (x-y,t-\tau)ds_y]d\tau\\
&&-\int_{z_1}^t M(\tau)[\int_{S_R} | \frac{\partial \Gamma}{\partial \eta _y} (x-y,t-\tau )| ds_y] d \tau, \quad x\in S_R, 0<z_1<t<T.\end{eqnarray*}
It is known that (see \cite{23}) there exist $C_0>0,$ such that $\Gamma$ satisfies 
 $$|\frac{\partial \Gamma}{\partial \eta_{y}}(x-y,t-\tau)|\le \frac{C_0}{(t-\tau )^\mu}\cdot \frac{1}{|x-y|^{(n+1-2\mu- \sigma ) }},\quad x,y \in S_R, ~\sigma\in(0,1).$$
Choose $1-\frac{\sigma}{2} < \mu <1,$ from Lemma \ref{pk}, there exist $C^*>0$ such that $$\int_{S_R}\frac{ds_y}{|x-y|^{(n+1-2\mu- \sigma ) }} <C^*.$$ Moreover, for $0<t_1<t_2$ and $t_1$ is closed to $t_2,$ there exists $c>0,$ such that
$$\int_{S_R} \Gamma (x-y,t_2-t_1)ds_y \ge  \frac{c}{\sqrt{t_2-t_1}},$$

Thus $$\frac{1}{2} M(t) \ge c \int_{z_1}^t \frac{e^{N^p(\tau)}}{\sqrt{t-\tau}}d\tau-C\int_{z_1}^t \frac{M(\tau)}{|t-\tau|^{\mu}}d \tau.$$
Since for $0<z_1< t_0< t <T,$ it follows that $M(t_0) \le M(t),$ thus the last equation becomes
\begin{equation*}\label{d8}
\frac{1}{2}M(t)\ge c \int_{z_1}^t \frac{e^{N^p(\tau)}}{\sqrt{T-\tau}}d\tau-C^*_1 M(t)|T-z_1|^{1-\mu}.
\end{equation*}
Similarly, for $0<z_2< t <T,$ we have
\begin{equation*}
\frac{1}{2}N(t)\ge c \int_{z_2}^t \frac{e^{M^q(\tau)}}{\sqrt{T-\tau}}d\tau-C^*_2 N(t)|T-z_2|^{1-\mu}.
\end{equation*}
Taking $z_1,z_2$ so that  $$C^*_1|T-z_1|^{1-\mu}\le1/2,\quad C^*_2|T-z_2|^{1-\mu}\le 1/2,$$ it follows 
\begin{equation}\label{mol}
M(t)\ge c \int_{z_1}^t \frac{e^{N^p(\tau)}}{\sqrt{T-\tau}}d\tau, \quad
N(t)\ge c \int_{z_2}^t \frac{e^{M^q(\tau)}}{\sqrt{T-\tau}}d\tau.
\end{equation}

Since both of $M,N$ increasing functions and from (\ref{Nobal}), we can find $T^*$ in $(0,T)$  such that  
$$M(t)\ge q^{\frac{1}{(q-1)}},\quad N(t)\ge p^{\frac{1}{(p-1)}}, \quad \mbox{for}\quad T^*\le t<T.$$  
 Thus
  $$e^{M^q(t)}\ge e^{qM(t)}, \quad e^{N^p(t)}\ge e^{pN(t)},\quad T^*\le t<T.$$  

 Therefore, if we choose $z_1,z_2$ in $(T^*,T),$ then (\ref{mol}) becomes 
 $$e^{M(t)}\ge c \int_{z_1}^t \frac{e^{pN(\tau)}}{\sqrt{T-\tau}}d\tau\equiv I_1(t), \quad
e^{N(t)}\ge c \int_{z_2}^t \frac{e^{qM(\tau)}}{\sqrt{T-\tau}}d\tau\equiv I_2(t).$$
Clearly, $$I_1^{'}(t)=c\frac{e^{pN(t)}}{\sqrt{T-t}}\ge \frac{cI_2^p}{\sqrt{T-t}},\quad I_2^{'}(t)=c\frac{e^{qM(t)}}{\sqrt{T-t}}\ge \frac{cI_1^q}{\sqrt{T-t}}.$$
By Lemma \ref{d}, it follows that 
\begin{equation}\label{ef} I_1(t)\le \frac{C}{(T-t)^{\frac{\alpha}{2}}},\quad I_2(t)\le \frac{C}{(T-t)^{\frac{\beta}{2}}},\quad t \in (\max\{z_1,z_2\},T).\end{equation}  
On the other hand, for $t^*=2t-T$ (Assuming that $t$ is close to $T$).
$$I_1(t)\ge c\int_{t^*}^t \frac{e^{pN(\tau)}}{\sqrt{T-\tau}}d\tau\ge c e^{pN(t^*)} \int_{2t-T}^t \frac{1}{\sqrt{T-\tau}}d\tau=2c(\sqrt{2}-1)\sqrt{T-t} e^{pN(t^*)}.$$
Combining the last inquality with (\ref{ef}) yields  
$$e^{N(t^*)} \le \frac{C}{2c(\sqrt{2}-1)(T-t)^{\frac{p+1}{2p(pq-1)}+\frac{1}{2p} } }= \frac{2^{\frac{q+1}{2(pq-1)}}C}{2c(\sqrt{2}-1)(T-t^*)^{\frac{q+1}{2(pq-1)} } }.$$
Thus, there exists a constant $c_1>0$ such that 
 \begin{equation*} e^{N(t^*)} (T-t^*)^{\frac{q+1}{2(pq-1)}} \le c_1.\end{equation*}
In the same way we can show \begin{equation*}e^{M(t^*)} (T-t^*)^{\frac{p+1}{2(pq-1)}}\le c_2.\end{equation*} This leads to, there exists $C_1,C_2>0$ such that 
  \begin{align}\label{yas}
  \max_{ \overline{B}_R}u(x,t)\le \log C_1-\frac{\alpha}{2} \log(T-t),\quad 0< t<T, \\ \label{sad}\max_{ \overline{B}_R}v(x,t)\le \log C_2-\frac{\beta}{2} \log (T-t),\quad 0< t<T.\end{align}  
\end{proof}
\section{Blow-up Set}
In order to show that the blow-up to problem (\ref{c1}) occurs only on the boundary, we need to recall the following lemma from \cite{18}.
\begin{lemma}\label{power}
Let $w$ is a continuous function on the domain $\overline{B}_R\times[0,T)$ and satisfies 
\begin{equation*} \left.\begin{array}{ll} 
w_t=\Delta w,&\quad (x,t) \in B_R \times (0,T),\\
w(x,t) \le \frac{C}{(T-t)^m},& \quad (x,t) \in S_R \times(0,T),\quad m>0.
\end{array} \right\} \end{equation*}
Then for any  $0<a<R$
$$\sup\{ w(x,t): 0 \le|x|\le a,~0\le t < T \} < \infty.$$
\end{lemma}
\begin{proof}
Set $$h(x)=(R^2-r^2)^2 ,~ r=|x|,$$ 
$$z(x,t)=\frac{C_1}{[h(x)+C_2(T-t)]^m}.$$ 

We can show that:
\begin{eqnarray*}\Delta h-\frac{(m+1)|\nabla h|^2}{h}&=& 8r^2-4n(R^2-r^2)-(m+1)16r^2\\ &\ge&-4nR^2-16R^2(m+1),\\
z_t-\Delta z&=&\frac{C_1m}{[h(x)+C_2(T-t)]^{m+1}}(C_2+\Delta h-\frac{(m+1)|\nabla h|^2}{h+C_2(T-t)})\\
&\ge& \frac{C_1m}{[h(x)+C_2(T-t)]^{m+1}}(C_2-4nR^2-16R^2(m+1)).\end{eqnarray*}
Let $$C_2=4nR^2+16R^2(m+1)+1$$ and take $C_1$ to be large such that $$z(x,0)\ge w(x,0),\quad x\in B_R.$$ Let $C_1\ge C(C_2)^m,$ which implies that $$z(x,t)\ge w(x,t) \quad \mbox{on}\quad S_R \times [0,T).$$
Then from the maximum principle \cite{21}, it follows that $$z(x,t)\ge w(x,t),\quad (x,t) \in \overline{B}_R \times (0,T)$$ and hence 
$$\sup\{w(x,t): 0\le|x|\le  a,0\le t <T\} \le C_1(R^2-a^2)^{-2m} < \infty, \quad 0\le a <R.$$
\end{proof}
 \begin{theorem}
 Let the assumptions of Theorem \ref{theorem d} be in force. Then $(u,v)$ blows up  only on the boundary.
 \end{theorem}
 \begin{proof}
 Using equations (\ref{yas}), (\ref{sad}) 
  $$u(R,t) \le \frac{c_1}{(T-t)^{\frac{\alpha}{2}}}, \quad v(R,t) \le \frac{c_2}{(T-t)^{\frac{\beta}{2}}},  \quad t \in (0,T).$$
  From Lemma \ref{power}, it follows that  $$\sup\{u(x,t): (x,t)\in B_a\times [0,T)\}\le C_1(R^2-a^2)^{-\alpha} <\infty,$$
  $$\sup\{v(x,t): (x,t)\in B_a\times [0,T)\}\le C_1(R^2-a^2)^{-\beta} <\infty,$$ for $a<R.$
   
 Therefore, $u,v$ blow up simultaneously and the blow-up occurs only on the boundary.
\end{proof}

\end{document}